\documentclass[10pt, a4paper]{article}
\usepackage[english]{babel}
\usepackage{amsmath}
\usepackage{amsthm}
\usepackage{amssymb}
\usepackage{bbm}
\usepackage{mathrsfs}
\usepackage{stmaryrd}

\newcommand{\arr}{\longrightarrow}

\newcommand{\imply}{\rightarrow}
\newcommand{\C}{\mathbb{C}}

\newcommand{\sub}{\textbf{Sub}}
\newcommand{\lb}{\llbracket}
\newcommand{\rb}{\rrbracket}

\newcommand{\ee}{\mathbb{\varepsilon}}

\usepackage[all]{xy}
\usepackage{bussproofs}

\mathcode`\:="603A     
\mathchardef\colon="303A  

\mathcode`\<="4268     
\mathcode`\>="5269     
\mathchardef\gt="313E  
\mathchardef\lt="313C  

\theoremstyle{definition}
\newtheorem{deff}{Definition}[section]
\newtheorem{prop}[deff]{Proposition}
\newtheorem{exa}[deff]{Example}
\newtheorem{lem}[deff]{Lemma}

\date{}
\begin{document}
\title{A categorical semantic for the Typed Epsilon Calculus}
\author{Fabio Pasquali}
\maketitle

\begin{abstract} We show that every boolean category satisfying AC provides a categorical semantic of the typed Epsilon calculus.
\end{abstract}

\section*{Introduction}\label{sec0}
Hilbert's Epsilon calculus is an extension of the Hilbert system for classical first order predicate logic by a term-forming operator. More specifically, for every well formed formula $\psi(x)$ there exists a term $\ee_{\psi}$, in which $x$ is not free, whose logical behavior is governed by the following axiom schema, also called Hilbert's transifite axiom$$\psi (x) \imply \psi [\ee_{\psi}/x]$$
Epsilon terms, i.e. terms of the form $\ee_{\psi}$, can be intuitively though as witnesses of the fact that $\psi (x)$ is true. This intuition can be made more precise if one consider that the transfinite axiom schema can be equivalently replaced by the following $$\exists x. \psi(x)\imply\psi [\ee_\psi/x]$$Moreover one can add the following extensionality axiom$$\forall x.(\psi(x)\leftrightarrow\phi(x))\rightarrow\ee_\psi=\ee_\phi$$which asserts that equivalent formulas have the same witness. When the extensionality axiom is considered, the calculus is called extensional Hilbert's Epsilon calculus.
A general account of the (extensional) Epsilon calculus can be found in \cite{HB1, HB2}.\\\\ In this paper we present a typed version of the Epsilon calculus and we show that it admits a simple categorical semantic.\\\\ In section \ref{sec2} we present the rules of the typed Epsilon calculus which we will be concerned with for the rest of the paper. Rules are given in the style of natural deduction, as it is customary in type-theoretic literature. In section \ref{sec1} we recall some ordinary facts concerning category theory. The section might be ignored by readers who are familiar with the notion of the Axiom of Choice, internal to a category, and with the notion of Boolean category. In section \ref{sec3} we recall the categorical semantic for first order calculi over a many typed signatures by mean of doctrines. We closely follow the approach given by Pitts in \cite{PittsCL}, which in turn is based on the approach originally due to Lawvere \cite{LawAdj}. A relevant difference with respect to \cite{PittsCL} is that we will deal with doctrines which serve as a semantic for classical calculi, whereas in \cite{PittsCL} the author focuses on the intuitionistic case. We give the definition of boolean doctrines and we introduce the notion of Epsilon doctrines. We conclude the section by showing that every Epsilon doctrine provides a categorical semantic for the typed Epsilon calculus presented in section \ref{sec2}. In section \ref{sec4} we define Epsilon categories, and we prove that every Epsilon category gives rise straightforwardly to an Epsilon doctrine. In the last section we consider Epsilon calculi with types constructors and we prove that there exist only trivial Epsilon doctrines that soundly interpret an Epsilon calculus with empty type.

\section{Typed Epsilon Calculus}\label{sec2}
Given a many typed signature $Sg$ and a denumerable set of variables, we shall denote types with capital letters ($A$, $B$, $A_1$, $A_2$\dots), while variables will be denoted with lower case latin letters ($x$, $y$, $x_1$, $x_2$\dots). Contexts, i.e finite lists of typed variables $(x_1:A_1, x_2:A_2,\dots,x_n:A_n)$, will be denoted by capital greek letters, while lower case greek letters will be used to denote formulas. Given two contexts $\Gamma$ and $\Theta$, we shall denote by $\Gamma,\Theta$ the concatenation of contexts.\\
\\
Term will be written in context, then we use the notation$$\Gamma \mid t:A$$
to express that $t$ is a well formed term of type $A$ in the context $\Gamma$. Similarly for formulas we will write $$\Gamma \mid \phi$$ to express that $\phi$ is a well formed formulas in the context $\Gamma$. Thus a sequent will take the form $$\Gamma\mid \phi_1,\dots,\phi_n\vdash\psi$$
We say that we have a calculus over a signature $Sg$ whenever $Sg$ is a many typed signature together with the standard rules of equational logic and structural rules of natural deduction. We refer the reader to \cite{PittsCL,Pra} for a detailed account of such rules.\\\\
A classical first order calculus over $Sg$ is a calculs over $Sg$ together with the standard formation, introduction and elimination rules of natural deduction for each one of the following propositional connectives $\wedge$, $\lor$, $\imply$, $\neg$, quantifiers $\exists, \forall$ and constants $\top$, $\bot$, where the rule of excluded middle is assumed to hold \cite{PittsCL,Pra}.
\begin{deff}\label{epsi}
Given a many typed signature $Sg$, an Epsilon Calculus over $Sg$ is a classical first order calculus over $Sg$ with additionally the following two rules
\[
\AxiomC{$\Gamma, x:A \mid \psi$}
\RightLabel{$\ee$-form}
\UnaryInfC{$\Gamma \mid \ee_\psi:A $}
\DisplayProof\ \ \ \ \ \ 
\AxiomC{$\Gamma, x:A \mid \psi$}
\RightLabel{$\ee$-I}
\UnaryInfC{$\Gamma \mid \exists x:A.\psi\vdash\psi[\ee_\psi/x]$}
\DisplayProof
\]
\end{deff}
The rule $\ee$-form ensures that for every formula $\Gamma, x:A\mid \psi$, there exists a term $\Gamma\mid\ee_\psi:A$ such that the substitution $\Gamma\mid\psi[\ee_\psi/x]$ is derivable. The rule $\ee$-I captures the Hilbert's idea that the substitution with Epsilon terms serves as existential quantification. In fact the equivalence $\Gamma\mid\psi(\ee_\psi)\dashv\vdash \exists x:A.\psi$ it is easily derivable.

\section{Preliminaries on categories}\label{sec1}
In this section we recall some standard definitions and known facts concerning categories. The reader who is not familiar with them is referred to \cite{elephant}.\\\\
In a given a category $\C$, for every object $A$, the collection of morphisms with codomain $A$ is canonically preordered by factorization, i.e. for morphisms $m:X\arr A $ and $n:Y\arr A$ we have that $m\le n$ whenever there exists $p:X\arr Y$ such that $np = m$. The same relation preorders also the collection of monomorphisms with codomain $A$. We shall denote the poset reflection of the latter by $\sub_\C(A)$.\\\\
If $\C$ is a finitely complete category then for every object $A$ of $\C$, the poset $\sub_\C(A)$ is an infsemillatice: for monomorphisms $m$ and $n$ with codomain $A$, the meet of $n$ and $m$ is represented by the pullback of $m$ along $n$, while the top element is represented by $id_A$. We shall denote the pullback of an arrow $g$ along an arrow $f$, both having the same codomain, by $f^*g$. Given a morphism $f:A\arr B$, pullbacking along $f$ gives raise to a functor $$f^*:\sub_\C(B)\arr\sub_\C(A)$$which is an infsemilattice homomorphism.\\\\
Recall that an epimorphism $e:A\arr B$ is regular if there exists a pair of parallel morphisms $f,g:X\arr Y$ such that the following
\[
\xymatrix{
X\ar@< 2pt>[r]^-{f}\ar@<-2pt>[r]_-{g}&A\ar[r]^-{e}&B
}
\]
is a coequalizer diagram.
\begin{deff}
A category $\C$ is regular, if
\begin{itemize}
\item[i)]$\C$ is finitely complete
\item[ii)] the class of regular epimorphisms and the class of monomorphisms form a factorization system. 
\item[iii)]regular epimorphisms are preserved by pullbacks
\end{itemize}
\end{deff}
By ii) we have that for every arrow $f:A\arr B$ there exists a regular epimorphism $e:A\arr Y$ and a monomorphism $m:Y\arr B$ that make the following diagram commute
\[
\xymatrix{
A\ar[rr]^-{f}\ar[rd]_-{e}&&B\\
&Y\ar[ru]_-{m}&
}
\]
If $\C$ is a regular category, then for every projection $\pi:X\times Y\arr Y$, the functor $\pi^*:\sub_\C(Y)\arr\sub_\C(X\times Y)$ has a left adjoint $\Sigma_\pi$ natural in $Y$. The condition of naturality, also called Beck Chevalley condition, is equivalent to require that for every morphism $f:Z\arr Y$ in $\C$, the following diagram commutes
\[
\xymatrix{
X\times Z\ar[r]^-{\Sigma_{\pi'}}&Z\\
X\times Y\ar[u]^-{(id_X\times f)^*}\ar[r]_-{\Sigma_\pi}&Y\ar[u]_-{f^*}
}
\]
where $\pi'$ is the projection $X\times Z\arr Z$ and $id_X\times f$ is the induced arrow from $X\times Z$ to $X\times Y$.\\\\ For a monomorphism $k$ with codomain $X\times Y$ we have that $\Sigma_\pi k$ is represented by the monomorphism in any factorization of $\pi k$.
\begin{deff}\label{coh}
A regular category $\C$ is coherent if for every object $X$ of $\C$ the infsemilattice $\sub_\C(X)$ has finite joins which are stable under pullback.
\end{deff}
We recall below the formulation of the Axiom of Choice and the Law of Excluded Middle internal to a category $\C$, we shall denote them by AC and LEM respectively.\\\\
AC: every epimorphism of $\C$ has a section. Which is to say that for every epimorphism $e:X\arr Y$ there exists a morphism $s_e:Y\arr X$ such that $es_e = id_Y$.\\\\
LEM: for every monomorphism $m:X\arr Y$ of $\C$, there exists a monomorphism $\neg m: \neg X\arr Y$, such that the domain of $m^*\neg m$ is an initial object and
\[
\xymatrix{
X\ar[r]^-{m}&Y&\neg X\ar[l]_-{\neg m}
}
\]
is a coproduct diagram. Which is to say that for every two arrows $f:X\arr Z$ and $g:\neg X\arr Z$ there exists a unique arrow $[f,g]:Y\arr Z$ with $[f,g]m = f$ and $[f,g]\neg m = g$.
\begin{deff}
A boolean category is a coherent category satisfying LEM.
\end{deff}
\begin{exa}\label{disco}By the known argument of Diaconescu \cite{diaco}, in every elementary topos LEM is a consequence of AC. Therefore every elementary topos satisfying AC is a boolean category\cite{elephant}.\end{exa}
If $\C$ is a boolean category, then for every object $A$ in $\C$, the infsemilattice $\sub_\C(A)$ is a boolean algebra. Moreover for every morphism $f:A\arr B$ the functor $f^*$ is a homomorphism of boolean algebras.

\section{Doctrines}\label{sec3}
In order to give a categorical semantic of an Epsilon Calculus, we will use the notion of doctrines. The definition of doctrine is based on the definition of hyperdoctrine given in \cite{PittsCL}, which in turn is based on the one originated by Lawvere in \cite{LawAdj}. Our notation and terminology are slightly different from those used in \cite{PittsCL}, where doctrines are called $prop$-$categories$.\\\\
We denote by \textbf{Boole} the category of Boolean algebras and homomorphisms between them. 
\begin{deff}\label{pippo}
A boolean doctrine is a pair $(\C,P)$ such that
\begin{itemize}
\item[i)]$\C$ is a category with finite products
\item[ii)] $P$ is a functor $P:\C^{op}\arr \textbf{Boole}$
\item[iii)] for every projection arrow $\pi:X\times Y\arr Y$, the homomorphism $P(\pi)$ has a left adjoint $\Sigma_\pi$ satisfying Beck Chevalley condition.
\end{itemize}
\end{deff}
For every object $A$ in $\C$, we shall denote by $\le$ the order over $P(A)$.
\begin{exa}\label{exa1} If $\C$ is a boolean category, then $(\C,\sub_\C)$ is a boolean doctrine, see section \ref{sec1}. Moreover if $\mathcal{C}$ is a full subcategory of $\C$ closed under finite products, then also $(\mathcal{C},\sub_\C)$ is a boolean doctrine.
\end{exa}
The following important proposition is due to Lawvere. We state it without giving a proof, which can be found in \cite{PittsCL}.
\begin{prop}\label{semi}
A classical first order calculus over a signature $Sg$ can be soundly interpreted in a boolean doctrine.
\end{prop}
The two features of the interpretation mentioned in proposition \ref{semi} that we will mainly concerned with are the semantic of substitution and quantification. We briefly recall them below.\\\\
The mentioned interpretation is provided by giving an object $\lb A \rb$ of $\C$ for every type of $Sg$ so that if $\Gamma$ is the context $(x_1:A_1,\dots,x_n:A_n)$, then $\lb \Gamma \rb$ is the product $\lb A_1\rb \times\dots\times\lb A_n\rb$. Up to an appropriate assignment of function symbols and relation symbols, the interpretation of a term $\Gamma\mid t:A$ is a morphism $\lb t\rb:\lb \Gamma\rb\arr\lb A\rb$ of $\C$, while the interpretation of a well formed formula $\Gamma\mid\phi$ is an element $\lb \phi\rb$ of $P(\lb \Gamma\rb)$. Then we have:\\\\
Substitution\begin{itemize}
\item[] given terms $\Gamma\mid t_i:A_i$, for $i=1,\dots,n$, and a formula $x_1:A_1,\dots,x_n:A_n\mid\phi$
$$\lb\phi(\vec{t}/\vec{x})\rb =P(<\lb t_1\rb,\dots,\lb t_n\rb>)(\lb \phi\rb)$$
\end{itemize}
Quantification
\begin{itemize}
\item[] given a formula $\Gamma,x:A\mid\phi$
$$\lb\exists x:A. \phi\rb = \Sigma_\pi\lb\phi\rb$$for $\pi$ the projection $\pi:\lb\Gamma\rb\times\lb A\rb\arr\lb\Gamma\rb$
\end{itemize}
\begin{deff}\label{gianna}
An Epsilon doctrine is a boolean doctrine $(\C,P)$ such that for every projection $\pi: X\times Y\arr X$ and every element $\psi$ in $P(X\times Y)$ there exists a morphism $\ee_\psi:X\arr Y$ in $\C$ such that $$\Sigma_\pi\psi\le P(<id_X,\ee_\psi>)(\psi)$$
\end{deff}
\begin{prop}\label{esound}
An Epsilon calculus over a many typed signature $Sg$ can be soundly interpreted in an Epsilon doctrine.
\end{prop}
\begin{proof} Suppose $(\C,P)$ is an Epsilon doctrine. Under the interpretation above, for a formula $\Gamma,x:A\mid\psi$ we have that $\lb\psi\rb$ is an element of $P(\lb\Gamma\rb\times\lb A\rb)$. Then there exists a morphism $\ee_{\lb\psi\rb}: \lb\Gamma\rb\arr\lb A\rb$. Define $\lb\ee_\psi\rb =\ee_{\lb\psi\rb}$.\end{proof}

\section{Epsilon Categories}\label{sec4}
In this section we define a class of categories which we shall call Epsilon categories. We will show that every Epsilon category gives rise to a epsilon doctrine.
\begin{deff}
An Epsilon category is a boolean category satisfying AC.
\end{deff}
\begin{exa}By \ref{disco}, an elementary topos verifying AC is an epsilon category.\end{exa}
In a category $\C$ with a terminal object $1$ an object $A$ is pointed if there exists a morphism $a:1\arr A$. We will denote by $\C_0$ the full subcategory of $\C$ on pointed objects.
\begin{lem}\label{pro} If $\C$ is a category with finite products, then $\C_0$ has finite products, moreover every projection $\pi$ in $\C_0$ has a section $s_\pi$. 
\end{lem}
\begin{proof}
If we have a morphism $b:1\arr B$ and $\pi:A\times B\arr A$ is a projection, then $s_\pi$ is
\[
\xymatrix{
A\ar[rr]^-{<id_A, b!_A>}&&A\times B
}
\]
where $!_A:A\arr 1$ is the unique arrow from $A$ to $1$.
\end{proof}
\begin{prop}\label{esound}
If $\C$ is an epsilon category, then $(\C_0, \textbf{Sub}_\C)$ is an epsilon doctrine.
\end{prop}
\begin{proof}From lemma \ref{pro} we know that $\C_0$ has finite products and since $\C$ is a boolean category we know from \ref{exa1} that $(\C_0, \sub_\C)$ is a boolean doctrine.\\\\ It remains to prove that for every two objects $X$ and $Y$ in $\C_0$ and for every monomorphism $\psi:A\arr X\times Y$ in $\C$, there exists a morphism $\ee_\psi:X\arr Y$ such that $$\Sigma_\pi \psi \le <id_X,\ee_\psi>^*\psi$$
for $\pi:X\times Y\arr X$ the first projection.\\\\Recall from section \ref{sec1} that $\Sigma_\pi\psi$ is represented by the monomorphism in the factorization of $\pi\psi$ as in the diagram below
\[
\xymatrix{
A\ar[rd]_-{e}\ar[rr]^{\pi\psi}&&X\\
&I\ar[ru]_-{m}&
}
\]
By AC, the regular epimorphism $e$ has a section $s_e$, moreover by lemma \ref{pro}, since $X$ and $Y$ belongs to $\C_0$, also $\pi:X\times Y\arr X$ has a section $s_\pi$. Now consider the following diagram
\[
\xymatrix{
&&Y&&\\
\neg I\ar[rr]^-{\neg ms_\pi}\ar@/_1.2pc/[rrd]_-{\neg m}&&X\times Y\ar[u]^-{\pi'}&&I\ar@/^1.2pc/[dll]^-{m}\ar[ll]_-{\psi s_e}\\
&&X\ar[u]_-{[ \psi s_e,\neg ms_\pi]}&&
}
\]
where by LEM the object $X$ is the coproduct of $I$ and $\neg I$ with canonical injections $m$ and $\neg m$, while $\pi'$ denotes the second projection. Define $$\ee_\psi = \pi'[\psi s_e,\neg ms_\pi]$$
and consider the diagram
\[
\xymatrix{
I\ar@/_1.2pc/[ddr]_-{m}\ar@/^1.2pc/[rrrd]^-{s_e}&&&\\
&P\ar[d]_-{p}\ar[rr]&&A\ar[d]^-{\psi}\\
&X\ar[rr]_-{<id_X,\ee_\psi>}&&X\times Y
}
\]
where the innermost square is a pullback of $\psi$ along $<id_X,\ee_\psi>$. Thus, by the universal property of pullbacks,  to prove that $m$ factors through $p$, and therefore that $\Sigma_\pi\psi \le <id_X,\ee_\psi>^*\psi$, it is enough to show that the outermost square commutes, which is true since
\begin{equation}\notag
\begin{split}
<id_X,\ee_\psi>m &= <m,\ee_\psi m> \\
& = <mes_e,\pi' [\psi s_e,\neg ms_\pi] m>\\
& = <\pi\psi s_e,\pi' \psi s_e>\\
&= \psi s_e
\end{split}
\end{equation}
\end{proof}
Thus every Epsilon category $\C$ provides an Epsilon doctrine $(\C_0,\sub\C)$. Moreover if we suppose that for every $A$ and every subobject $s$ over $A$ we have a choice of a representative, i.e. a monomorphism $c_s:X\arr A$ such that $[c_s] = s$, then the following rule
\[
\AxiomC{$\Gamma, x:A \mid \psi\vdash\phi$}
\AxiomC{$\Gamma, x:A \mid \phi\vdash\psi$}
\RightLabel{$\ee$-ex}
\BinaryInfC{$\Gamma \mid \ee_\psi = \ee_\phi:A $}
\DisplayProof\]
can be soundly interpreted in $(\C_0,\sub_\C)$, by defining $\lb\ee_\psi\rb =\ee_{c_{\lb\psi\rb}}$.\\\\The rule $\ee$-ex asserts that the epsilon terms that correspond to equivalent formulas are equal. If we add $\ee$-ex the the Epsilon calculus we obtain the typed version of extensional  Hilbert's Epsilon calculus presented in \cite{HB2}.\\\\
Note also that, in order to prove that $(\C,\sub\C)$ is an Epsilon doctrine for $\C$ an Epsilon category, we made use of AC. The following proposition shows that this was necessary.
\begin{prop}\label{ciccio} If $(\C, \sub_\C)$ is an Epsilon doctrine, then AC holds in $\C$.
\end{prop}
\begin{proof}
Suppose $f:X\arr Y$ is an epimorphism, and consider the subobject represented by $<f,id_Y>:X\arr Y\times X$. Since $f$ is an epimorphism, $\Sigma_\pi <f,id_Y>$ is represented by $id_Y$. By assumption there exists a morphism $k$ making the following commute
\[
\xymatrix{
Y\ar[dr]_-{id_Y}\ar[r]^-{k}&P\ar[d]\ar[rr]^-{s}&&X\ar[d]^-{<f,id_Y>}\\
&Y\ar[rr]_-{<id_Y, \ee_{<f,id_Y>}>}&&Y\times X
}
\]
where the right square is a pullback. Thus $sk$ is a section of $f$.\end{proof}

\section{Epsilon Calculus with types constructors}\label{sec5}
In this section we consider the case in which the underling signature of an Epsilon calculus has types constructors. In particular we will deal with product types, function types, sum types and the empty type. For the rules concerning each types constructor we refer the reader to \cite{PittsCL}.\\\\ After the work of Lambek and Scott we know that a first order calculus over a signature with product types and functions types, can be soundly interpreted in a doctrine $(\C,P)$ where $\C$ is cartesian closed \cite{LamScott, PittsCL}. On the other hand, a first order calculus over a signature with sum types, can be soundly interpreted in a doctrine $(\C,P)$ where $\C$ has binary coproducts \cite{PittsCL}.\\\\If $\C$ is cartesian closed, its full subcategory $\C_0$ on pointed objects is cartesian closed: if $B$ is in $\C_0$, i.e there exists $b:1\arr B$, then for every object $A$ we have that the exponential transpose of
\[
\xymatrix{
1\times A\ar[r]^-{\pi}&1\ar[r]^{b}&B
}
\]
points $B^A$. Moreover if $B$ is pointed, $A+B$ is pointed by the arrow $i_Bb:1\arr A+B$, where $i_B:B\arr A+B$ is the canonical injection.
Thus if $\C$ is a cartesian closed epsilon category with binary coproducts, such as every elementary topos satisfying AC, then $(\C_0, \sub_\C)$ is an Epsilon doctrine which soundly interprets an  Epsilon calculus over a signature with product types, function types and sum types.\\\\
As one might expect, the case of the empty type is not as straightforward as the cases considered above. In fact, under the interpretation given in section \ref{sec3}, there are no Epsilon doctrines which soundly interpret an Epsilon calculus with empty type, except trivial ones, as specified by the following proposition 
\begin{prop}
Let $Sg$ be a signature with empty type. If an Epsilon calculus over $Sg$ can be soundly interpreted in an Epsilon doctrine $(\C,P)$, then every object of $\C$ is terminal. 
\end{prop}
\begin{proof}
Since $Sg$ has an empty type, then $\C$ has a stable initial object 0 \cite{PittsCL}. Moreover for the top element $\top$ in $P(1\times 0)$ there exists in $\C$ the arrow $\ee_\top:1\arr 0$ which is necessarily an isomorphism. Then $A \simeq 1\times A \simeq 0\times A \simeq 0\simeq 1$ for every object $A$ in $\C$.
\end{proof}

\end{document}